\documentclass{article}
\usepackage{amssymb,amsmath}
\usepackage{graphicx}
\usepackage{CJK}

\def\qed{\hfill {\hbox{${\vcenter{\vbox{               
   \hrule height 0.4pt\hbox{\vrule width 0.4pt height 6pt
   \kern5pt\vrule width 0.4pt}\hrule height 0.4pt}}}$}}}

\def\tr{\triangleright}
\def\bar{\overline}

\newtheorem{theorem}{Theorem}
\newtheorem{definition}{Definition}

\newtheorem{proposition}[theorem]{Proposition}
\newtheorem{corollary}[theorem]{Corollary}
\newtheorem{example}{Example}

\newtheorem{remark}{Remark}

\newenvironment{proof}[1][Proof]{\smallskip\noindent{\bf #1.}\quad}%
{\qed\par\medskip}

\date{}

\title{\Large \textbf{Bikei, Involutory Biracks and unoriented link invariants}}

\author{Sinan Aksoy\footnote{Email:\ sga@uchicago.edu} \and 
Sam Nelson\footnote{Email:\ knots@esotericka.org}}

\begin{document}
\maketitle

\begin{abstract}
We identify a subcategory of biracks which define counting invariants of 
unoriented links, which we call \textit{involutory biracks}. In particular,
involutory biracks of birack rank $N=1$ are biquandles, which we call 
\textit{bikei} or \begin{CJK*}{UTF8}{min}双圭\end{CJK*}. We define counting 
invariants of unoriented classical and virtual links using finite 
involutory biracks, and we give an example of a non-involutory birack 
whose counting invariant detects the non-invertibility of a virtual knot.
\end{abstract}

\medskip

\quad
\parbox{5.5in}{
\smallskip

}

\section{\large\textbf{Introduction}}

Much attention has recently been focused on the study of invariants of oriented
knots and links defined using algebraic objects known as \textit{biquandles},
solutions to the set-theoretic Yang-Baxter equation satisfying certain
invertibility conditions \cite{CES2, CS, FJK, KR}. Every oriented classical or 
virtual knot has a \textit{fundamental biquandle} whose isomorphism class 
is a strong invariant -- indeed, a complete invariant of classical knots 
when considered up to ambient homeomorphism.

Comparing isomorphism classes of fundamental biquandles directly is 
generally impractical, so for more practical biquandle-derived invariants
we can either look to functorial invariants like the Alexander and
quaternionic biquandle polynomials studied in \cite{BF1,BF2,BF3,BF4} which 
generalize the
classical Alexander polynomial, or to representational invariants such as
the counting invariant $\Phi^{\mathbb{Z}}_{X}(L)=|\mathrm{Hom}(FB(L),X)|$
where $X$ is a finite biquandle \cite{CN3,NV}.

\textit{Quandles} are a special case of biquandles. Of particular
interest are the \textit{CJKLS quandle 2-cocycle} invariants of oriented
classical and virtual knots and links defined in terms of homomorphisms
from the fundamental quandle of a knot to a finite quandle, enhanced by
a \textit{Boltzmann weight} defined from an element of the second
cohomology of the finite quandle in question \cite{CJKLS}. The study of 
quandle homology has recently turned to \textit{involutory quandles}, 
also known as \textit{kei} or \begin{CJK*}{UTF8}{min}圭\end{CJK*}; these 
are the type of quandles suitable for defining invariants of unoriented
knots and links \cite{NP}. Kei were considered as far back as 1945 \cite{T}.

\textit{Racks} are the objects analogous to quandles which are appropriate 
for defining representational invariants of blackboard-framed oriented
knots and links \cite{FR}. Racks are a special case of \textit{biracks},
recently studied in papers such as \cite{BF2} and \cite{N}.

In this paper we generalize the kei idea to the setting of biracks,
defining counting invariants for unoriented framed and unframed 
knots and links. In section 
\ref{BK} we identify the necessary and sufficient conditions for a
birack to be involutory, and we give examples of involutory biracks,
involutory racks, and involutory biquandles (also known as \textit{bikei} or
\begin{CJK*}{UTF8}{min}双圭\end{CJK*})
as well as a schematic map of the various types of 
biracks associated to categories of knots and links. In section \ref{CE} we
define counting invariants associated to involutory biracks and
give some computations and examples, including a biquandle whose counting
invariant distinguishes a virtual knot from its inverse, answering a question
posed to the second author by Xiao-Song Lin. In section \ref{Q} we list
some open questions for future research.

\section{\large\textbf{Bikei and Involutory Biracks}}\label{BK}

We begin with a definition. (See \cite{FJK,N}).

\begin{definition}
\textup{A \textit{birack} $(X,B)$ is a set $X$ with a map 
$B:X\times X\to X\times X$ which satisfies}
\begin{itemize}
\item \textup{$B$ is invertible, i.e there exists a map 
$B^{-1}:X\times X\to X\times X$
satisfying $B\circ B^{-1}=\mathrm{Id}_{X\times X}=B^{-1}\circ B$,}
\item \textup{$B$ is \textit{sideways invertible}, i.e there exists a unique 
invertible map $S:X\times X\to X\times X$ satisfying}
\[S(B_1(x,y),x)=(B_2(x,y),y),\]
\textup{for all $x,y\in X$,}
\item \textup{The sideways maps $S$ and $S^{-1}$ are \textit{diagonally 
bijective}, i.e. the compositions $S_1^{\pm 1}\circ\Delta$ and
$S_2^{\pm 1}\circ\Delta$ of the components of $S^{\pm 1}$ with 
the diagonal map $\Delta(x)=(x,x)$ are bijections, and} 
\item \textup{$B$ is a solution to the 
\textit{set-theoretic Yang-Baxter equation}:}
\[(B\times \mathrm{Id})\circ(\mathrm{Id}\times B)\circ(B\times \mathrm{Id})=
(\mathrm{Id}\times B)\circ(B\times \mathrm{Id})\circ(\mathrm{Id}\times B)\]
\end{itemize}
\textup{The components of $B$ and $B^{-1}$ are sometimes written with the 
alternate notation $B(x,y)=(y^x,x_y)$ and $B^{-1}(x,y)=(x_{\bar{y}},y^{\bar{x}})$.}
\end{definition}

The birack axioms are motivated by the oriented Reidemeister moves, where we
interpret the map $B$ as a map of semiarc labels going through a crossing:
\[\includegraphics{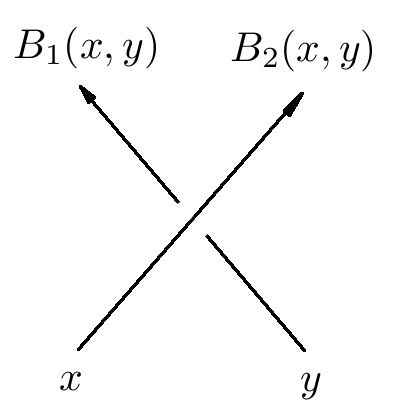} \quad \includegraphics{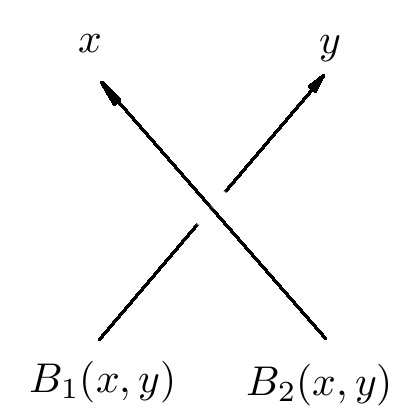}
\quad \includegraphics{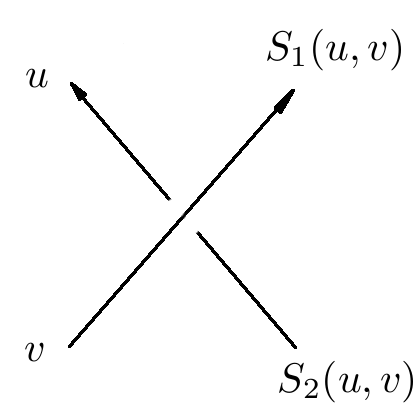}\]

As shown in \cite{N}, sideways invertibility defines bijections
$\alpha:X\to X$ and $\pi:X\to X$ defined by $\alpha=(S_2^{-1}\circ\Delta)^{-1}$
and $\pi=S_1^{-1}\circ\Delta\circ \alpha$ which give the labels of semiarcs in 
a blackboard-framed type I move:

\[\includegraphics{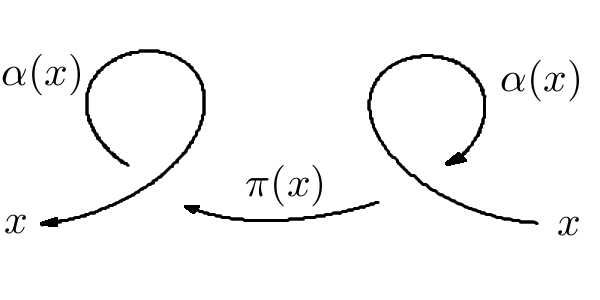}\]

The bijection $\pi(x)$ is known as the \textit{kink map}, and its exponent
$N$, i.e. the smallest integer $N$ such that $\pi^N(x)=x$ for all $x\in X$,
is known as the \textit{birack rank} or \textit{birack characteristic} of
$X$.

We would like to modify the birack axioms to remove the orientation 
requirement with the goal of obtaining invariants of unoriented links.
We will use the convention that if a crossing is positioned with the 
understrand on the right, then the upward map will be our $B$ map:
\[\includegraphics{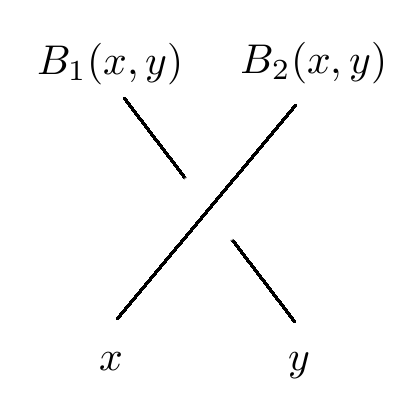}\]

We first observe that after rotating the crossing by 
$180^{\circ}$, we have $B(x,y)=(u,v)$ implies $B(v,u)=(y,x)$:
\[\includegraphics{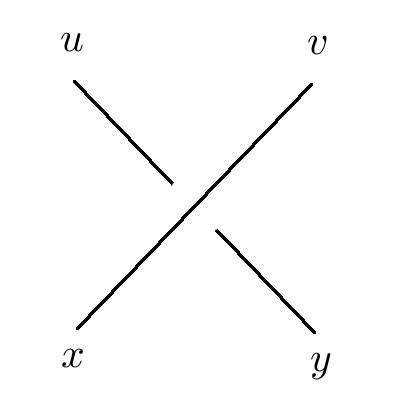} \quad \includegraphics{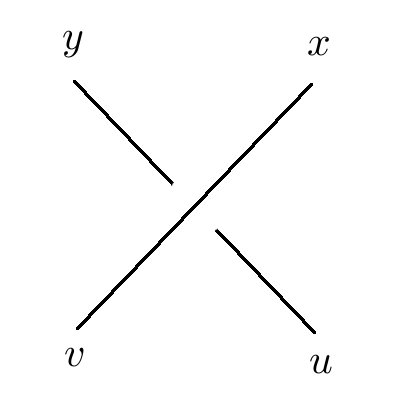}\]
For any set $X$, let 
$\tau:X\times X\to X\times X$ be the map $\tau(x,y)=(y,x)$. Then 
$B(v,u)=(y,x)$ implies $\tau\circ B \circ\tau (u,v)=(x,y),$ and we have
\[B^{-1}=\tau \circ B\circ\tau\]
or equivalently $(\tau\circ B)^2=\mathrm{Id}$.

The Reidemeister II move then requires that the upward map
at a crossing with the understrand on the left is $B^{-1}=\tau \circ B\circ 
\tau$, as well as that the sideways map $S$ is $B^{-1}$:
\[\includegraphics{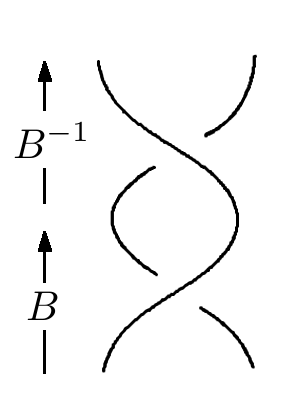}\quad \quad \quad
\includegraphics{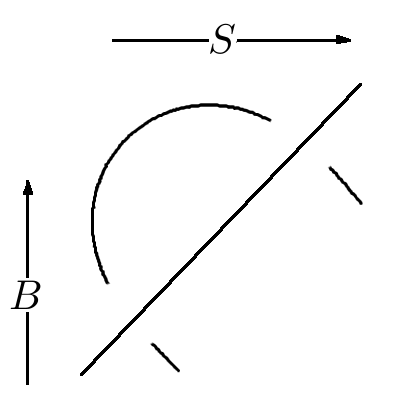} 
\]

With these observations, we can now define what it means for a birack to
be involutory.

\begin{definition}
\textup{An \textit{involutory birack} $(X,B)$ is a set $X$ with a map 
$B:X\times X\to X\times X$ which satisfies}
\begin{itemize}
\item $(\tau\circ B)^2=\mathrm{Id}$ \textup{where} $\tau(x,y)=(y,x)$,
\item \textup{The compositions $B^{\pm 1}_1\circ\Delta$ and
$B^{\pm 1}_2\circ\Delta$ of the diagonal map 
$\Delta(x)=(x,x)$ with the components of $B^{\pm 1}$ are bijections, and} 
\item \textup{$B$ is a solution to the 
\textit{set-theoretic Yang-Baxter equation}:}
\[(B\times \mathrm{Id})\circ(\mathrm{Id}\times B)\circ(B\times \mathrm{Id})=
(\mathrm{Id}\times B)\circ(B\times \mathrm{Id})\circ(\mathrm{Id}\times B).\]
\end{itemize}
\textup{An involutory birack with birack rank $N=1$ is an \textit{involutory 
biquandle} or \textit{bikei} (\begin{CJK*}{UTF8}{min}双圭\end{CJK*}).}
\end{definition}

\begin{remark}\label{rmk1}
\textup{The term ``involutory'' refers to the fact that the maps 
$u_x,l_x:X\to X$
defined by $u_x(y)=B_1(x,y)$ and $l_x(y)=B_2(y,x)$ are involutions, i.e. 
$u_x^2=\mathrm{Id}=l_x^2$ for all $x\in X$.}
\[\includegraphics{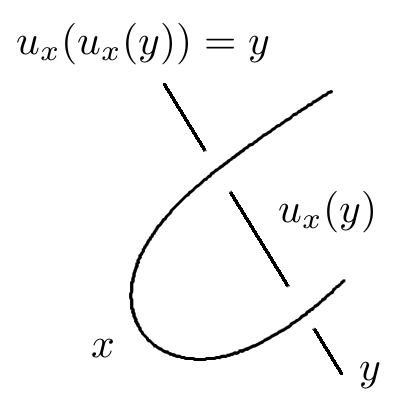}\quad \includegraphics{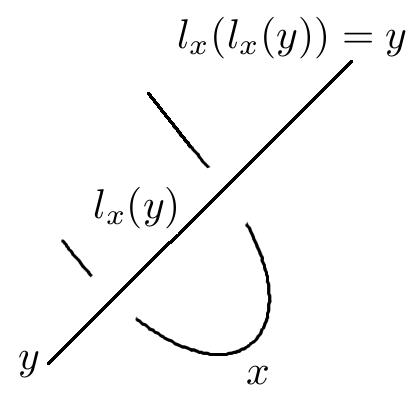}\]
\end{remark}

\begin{example}\label{ex1}
\textup{Recall from \cite{N} that any set $X$ has the structure
of a birack defined by $B(x,y)=(\sigma(y),\rho(x))$ where $\sigma,\rho:X\to X$
are commuting bijections; these are known as \textit{constant action biracks}. 
Such a birack is involutory iff $\rho^2=\sigma^2=\mathrm{Id}$, since}
\[\tau \circ B\circ \tau \circ B(x,y)=(\rho^2(x),\sigma^2(y)).\]
\end{example}

\begin{example}
\textup{A birack in which $B_2(x,y)=x$ is a \textit{rack}. A rack is then 
involutory iff $B_1(x,B_1(x,y))=y$ for all $x,y\in X$. Alternate notations for
the rack operation $B_1(x,y)$ include $y\tr x$ and $y^x$; using these 
conventions, a rack is involutory iff $(y\tr x)\tr x=y$ or $(y^x)^x=y$ for 
all $x,y\in X$. A rack of birack rank $N=1$ is known as a \textit{quandle}; 
involutory quandles are also known as \textit{kei}.}
\end{example}

We summarize the relatonships between these objects with the following Venn 
diagram. Note that the sizes of the circles are not meant to reflect 
proportions.
\[\scalebox{0.8}{\includegraphics{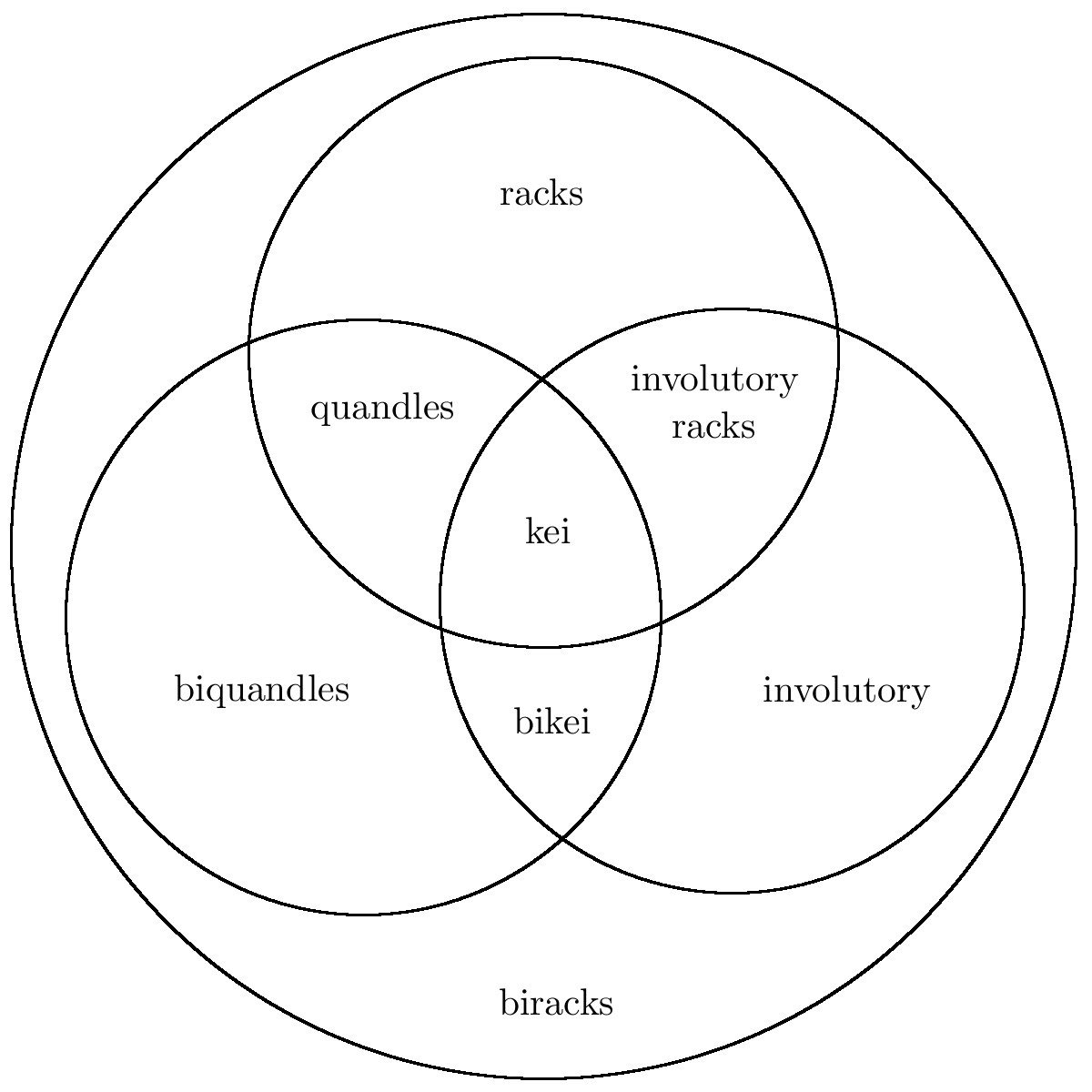}}\]

Recall from \cite{N} that any module over the ring
$\tilde\Lambda=\mathbb{Z}[t^{\pm 1,s,r^{\pm 1}}]/(s^2-(1-tr)s)$  has the 
structure of a birack defined by $B(x,y)=(sx+ty,rx)$, known as a 
\textit{$(t,s,r)$-birack.} The kink map of such a birack is $\pi(x)=(tr+s)x$, 
so the birack rank of a $(t,s,r)$-birack is the smallest integer $N$ such that
 $(tr+s)^N=1$. A $(t,s,r)$-birack in which $r=1$ is rack known as a 
\textit{$(t,s)$-rack} \cite{FR,N,CN1}; a $(t,s,r)$-birack with rank $N=1$
is an \textit{Alexander biquandle}, and if we have both $r=1$ and $N=1$
we have an \textit{Alexander quandle}.

\begin{proposition}
\textup{A $(t,s,r)$-birack is involutory iff we have $t^2=r^2=1$ and 
$(t+r)s=(1-r)s=0$.}
\end{proposition}

\begin{proof}
\begin{eqnarray*}
\tau \circ B\circ \tau \circ B (x,y) & = & \tau \circ B\circ \tau (sx+ty,rx) \\
& = & \tau \circ B (rx,sx+ty) \\
& = & \tau (srx+t(sx+ty),r^2x) \\
& = & (r^2x, t^2y+(t+r)sx)
\end{eqnarray*}
so setting $(\tau \circ B)^2=\mathrm{Id}$ we obtain $t^2=r^2=1$ and $(t+r)s=0$.

Finally, note that in a $(t,s,r)$-birack we have $S(u,v)=(rv,t^{-1}u-t^{-1}sv)$.
Then 
\begin{eqnarray*}
S\circ B(x,y) & = & S(sx+ty,rx) \\
& = & (r^2x,t^{-1}(sx+ty)-t^{-1}srx) \\
& = & (r^2x,y+(t^{-1}s-t^{-1}sr)x) \\
\end{eqnarray*} 
so $S\circ B=\mathrm{Id}$ implies $r^2=1$ and $t^{-1}s-t^{-1}sr=0$,
and multiplication by $t$ reduces the latter to $(1-r)s=0$.
\end{proof}

\begin{corollary}
A $(t,s)$-rack is involutory if and only if $t^2=1$ and $(t+1)s=0$.
\end{corollary}

\begin{corollary}
An Alexander biquandle is involutory if and only if $t^2=r^2=1$ and 
$(1-t)(1-r)=0$.
\end{corollary}

\begin{proof}
In an Alexander biquandle, we have $tr+s=1$ which implies $s=1-tr$. Then
\[(t+r)s=(t+r)(1-tr)=t-t^2r+r-tr^2\]which is zero provided $t^2=r^2=1$.
The condition that $(1-r)s=0$ is then
\[(1-r)(1-tr)=1-r-tr+tr^2=1-r-tr+t=1-r+t(1-r)=(1-t)(1-r)=0,\]
as required.
\end{proof}

These together imply the well-known result:
\begin{corollary}
An Alexander quandle is involutory if and only if $t^2=1$.
\end{corollary}

Let $X=\{x_1\dots, x_n\}$ be a finite set. We can specify an involutory 
birack structure on $X$ with a pair of $n\times n$ matrices specifying 
the operation tables of $y^x=B_1(x,y)$ and $x_y=B_2(x,y)$, i.e. 
$M_{(X,B)}=[U|L]$ where $U(i,j)=k$ and $L(i,j)=h$ where $x_k=B_1(x_j,x_i)$ and
$x_h=B_2(x_i,x_j)$. This allows us to do computations with biracks for 
which we lack convenient formulas.

\begin{example}
\textup{Let $X=\mathbb{Z}_n=\{1,2,3,4\}$. We can give $X$ the structure of an 
involutory  $(t,s,r)$-birack
on $X$ by choosing invertible elements $t,r\in \mathbb{Z}_n^{\ast}$ and an 
element $s\in\mathbb{Z}_n$ satisfying the conditions $s^2=(1-tr)s$, $t^2=s^2=1$ 
and $(1-r)s=(t+r)s=0$. For example, $X=\mathbb{Z}_4$ becomes an involutory
$(t,s,r)$-birack by setting $s=2$, $t=1$ and $r=3$. Then we have 
$s^2=4=0=(1-1(3))(2)$, $t^2=1$, $r^2=9=1$, $(t+r)s=(1+3)(2)=0$ and 
$(1-r)s=(1-3)2=0$. The birack matrix is given by}
\[M_{(X,B)}=\left[\begin{array}{rrrr|rrrr}
3 & 1 & 3 & 1  & 3 & 3 & 3 & 3 \\
4 & 2 & 4 & 2  & 2 & 2 & 2 & 2 \\
1 & 3 & 1 & 3  & 1 & 1 & 1 & 1 \\
2 & 4 & 2 & 4  & 4 & 4 & 4 & 4 \\
\end{array}\right].\]
\end{example}

\begin{remark}
\textup{The columns of the birack matrix are the images of the maps $u_x$ and 
$l_x$ mentioned in remark \ref{rmk1}. Hence, a birack whose matrix contains any 
column not representing an involution is not involutory.}
\end{remark}

\begin{example}
\textup{Let $K$ be a blackboard-framed classical or virtual knot or link 
diagram. The \textit{fundamental involutory birack} of $K$, denoted $IB(K)$, 
is the set of equivalence classes of involutory birack words modulo the 
equivalence relation generated by the involutory birack axioms and the 
crossing relations in $K$. More precisely, let:
\begin{itemize}
\item $G$, the set of \textit{generators}, correspond bijectively with the
set of semiarcs in $k$,
\item $W(G)$, the set of \textit{involutory birack words in G}, be defined 
inductively by the rules that (1) $G\subset W(G)$ and (2) $B(g,h)\in W(G)$ 
for all  $g,h\in W(G)$, and 
\item $\sim$ be the smallest equivalence relation on $W(G)$ containing each 
of the crossing relations together with 
\[(\tau \circ B)^2(g,h)\sim(g,h)\ \mathrm{and}\ 
(B\times \mathrm{Id})(\mathrm{Id}\times B)(B\times \mathrm{Id})(g,h,k)
\sim(\mathrm{Id}\times B)(B\times \mathrm{Id})(\mathrm{Id}\times B)(g,h,k)\]
for all $g,h,k\in W(G)$.
\end{itemize}
We will specify the fundamental involutory birack with a list of generators 
$G$ and crossing relations $R$, $IB(K)=\langle G|R\rangle$ with the birack 
axiom relations understood. For example, the trefoil knot below has the listed 
fundamental involutory birack presentation:}
\[\includegraphics{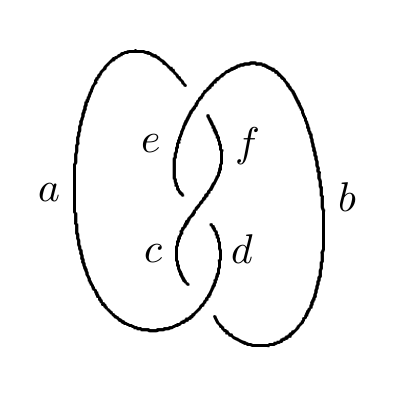} \quad \raisebox{0.75in}{$IB=\langle a,b,c,d,e,f \ | \
B(a,b)=(c,d), B(c,d)=(e,f), B(e,f)=(a,b)\rangle.$}\]
\textup{For virtual knots and links, we ignore virtual crossings, with semiarcs going 
from one classical over or undercrossing point to the next:}
\[\includegraphics{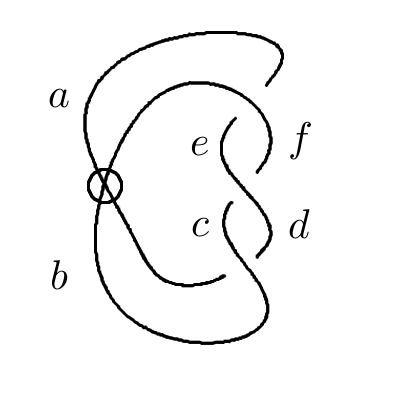} \quad \raisebox{0.75in}{$IB=\langle a,b,c,d,e,f \ | \
B(a,b)=(c,d), B(c,d)=(e,f), B(e,f)=(b,a)\rangle.$}\]
\end{example}

\begin{remark}
\textup{The fundamental involutory birack is analogous to the 
\textit{fundamental birack} $BR(L)$ of a knot or link (see \cite{FJK} 
or \cite{N}). Indeed, there is a functor $I:\mathbf{Br}\to \mathbf{IBr}$ 
from the category of finitely generated biracks to the category of finitely 
generated involutory biracks defined by setting $B^{-1}=\tau B\tau=S$ in
presentations of the objects of $\mathbf{Br}$ in addition to the 
inclusion functor $\mathbf{IBr}\to \mathbf{Br}$.}
\end{remark}

As with other algebraic structures, we have the following standard definitions:
\begin{definition}\textup{
A map $f:X\to Y$ between involutory biracks is a \textit{homomorphism}
if for all $x,y\in X$ we have}
\[B(f(x),f(y))=(f(B_1(x,y)),f(B_2(x,y))).\]
\end{definition}
\begin{definition}\textup{A subset $Y\subset X$ of an involutory birack $X$
is a \textit{subbirack} if $B(Y\times Y)\subset Y\times Y$.}
\end{definition}

\section{\large\textbf{Invariants of Unoriented Links}}\label{CE}

We begin this section by recalling the counting invariant of oriented 
classical and virtual links associated to a finite involutory birack 
defined in \cite{N}.
Let $X$ be a finite set and $B:X\times X\to X\times X$ a birack structure 
on $X$. Let $L$ be an oriented link diagram (classical or virtual) with $c$ 
components. A framing of $L$ is given by an element $\mathbf{w}\in 
\mathbb{Z}^c$, where the $k$th entry of $\mathbf{w}$ gives the 
\textit{writhe} of the $k$th component of $L$, i.e. the sum of the crossing
signs of the crossings where both strands are from component $k$ using the
convention below.
\[\includegraphics{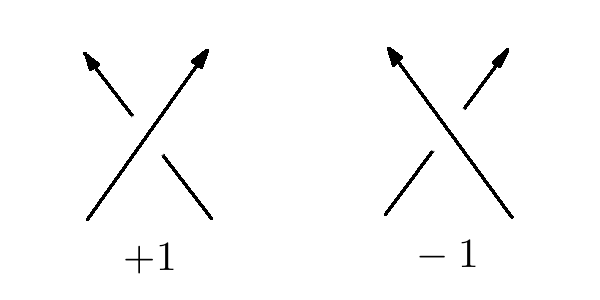}\]

A \textit{birack labeling} of $L$ by $X$ is a homomorphism $f:BR(L)\to X$
from the fundamental birack of $L$  to $X$. 
In particular, every homomorphism $f:BR(L)\to X$ assigns an element
of $X$ to each generator of $BR(L)$ and hence to each semiarc in $L$, and such
an assignment determines a homomorphism if and only if the crossing relations
are satisfied in $X$ by the assignment. Thus, we can visualize birack 
homomorphisms $f:BR(L)\to X$ as labelings of the semiarcs in a diagram of 
$L$ by their images in $X$.

By construction, changing a diagram by the \textit{blackboard-framed 
Reidemeister moves} 
\[\includegraphics{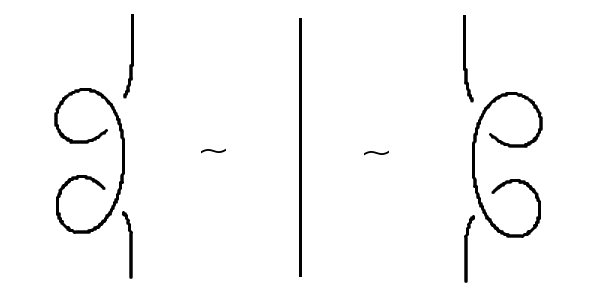}\quad \includegraphics{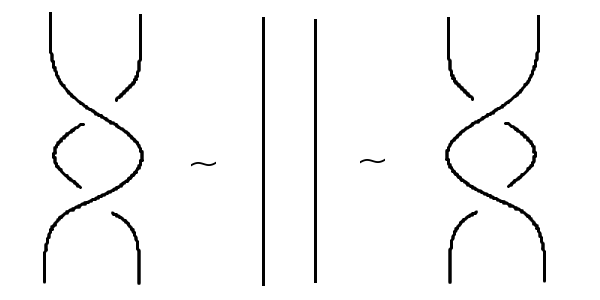}
\quad \includegraphics{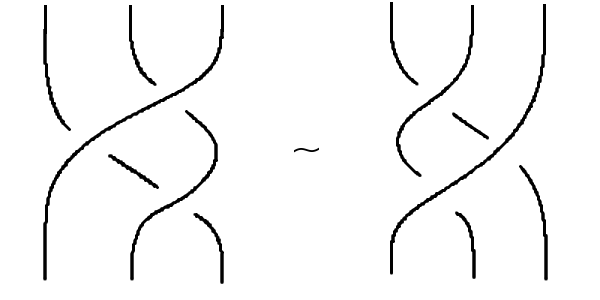}\]
induces a bijection on the sets of birack
homomorphisms. In particular, the number of birack labelings of a link
diagram by a finite birack $X$ is an integer-valued invariant of 
blackboard-framed isotopy. 

Now, let $N$ be the birack rank of $X$. For any birack labeling of $L$ by $X$,
there is a unique corresponding birack labeling of any framed link diagram
related to $L$ by blackboard framed Reidemeister moves together with the
\textit{$N$ phone cord move}:
\[\includegraphics{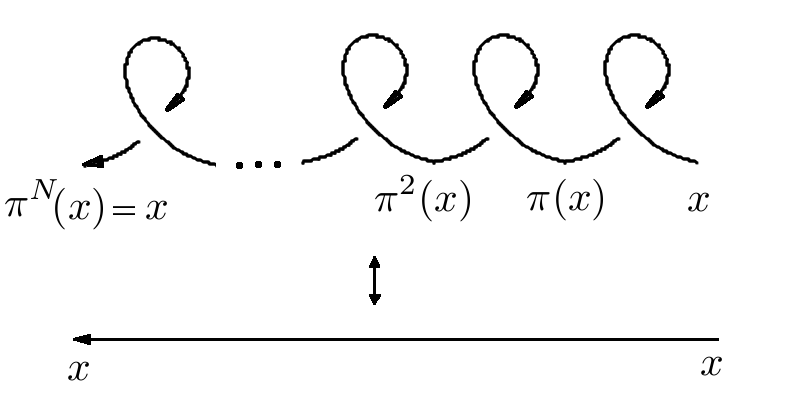}\]

Thus, the numbers of labelings $|\mathrm{Hom}(RB(L),X)|$ are periodic in the 
writhe of each component with period $N$, and we obtain an invariant of 
unframed isotopy by summing these numbers of labelings over a complete
period of writhes:
\[\Phi_{(X,B)}^{\mathbb{Z}}(L)=\sum_{\mathbf{w}\in(\mathbb{Z}_N)^c}
|\mathrm{Hom}(BR(L,\mathbf{w}),X)|\]
is the \textit{integral birack counting invariant}, defined in \cite{N}.

This same definition applies unmodified 
in the involutory birack case with unoriented links: while crossing signs are 
undefined for intercomponent crossings in unoriented links, they are 
well defined for intracomponent crossings since both choices of orientation
of a given component determine the same sign for each intracomponent crossing.
Thus, we have:

\begin{definition}
Let $L$ be an unoriented classical or virtual link with $c$ components and let
$(X,B)$ be a finite involutory birack. The \textit{integral involutory birack 
counting invariant} is
\[\Phi_{(X,B)}^{\mathbb{Z}}(L)=\sum_{\mathbf{w}\in(\mathbb{Z}_N)^c}
|\mathrm{Hom}(IB(L,\mathbf{w}),X)|.\]
\end{definition}

\begin{example}
\textup{The well-known Fox 3-coloring invariant is a special case of
the involutory birack counting invariant. Specifically, let $X=\mathbb{Z}_3$
with $t=2$, $s=2$ and $r=1$; then $s^2=1=(1-2(1))2=(1-tr)s$, so we have a 
$(t,s,r)$-birack. Moreover, $t^2=r^2=1$, $(t+r)s=3(2)=0$
and $(1-r)s=(1-1)2=0$, so $X$ is involutory, and $t+s=4=1$, so $X$ is a kei.
As a labeling rule, we have}
\[\includegraphics{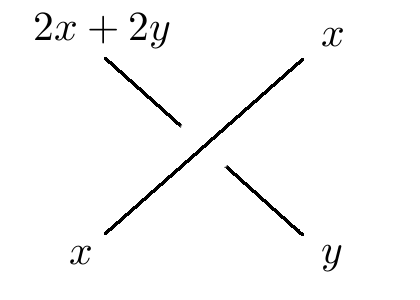}\]
\textup{which amounts to ``all three colors agree or all three are distinct''.
The fact that $\Phi_{(X,B)}^{\mathbb{Z}}(3_1)=9\ne 3=
\Phi_{(X,B)}^{\mathbb{Z}}(\mathrm{Unknot})$ is perhaps the easiest proof that the
trefoil is nontrivially knotted.}
\end{example}

An \textit{enhancement} of the birack counting invariant assigns a 
blackboard-framed and $N$-phone cord invariant signature to each birack 
labeling or homomorphism in $\mathrm{Hom}(BR(L),X)$; the multiset of such
signatures over a complete set of writhe vectors is then an enhanced 
invariant which determines the counting invariant value but is generally
stronger. All of the enhancements of $\Phi_{(X,B)}^{\mathbb{Z}}(L)$ defined
in \cite{N} are also defined for involutory biracks. These include:
\begin{itemize}
\item \textit{The image-enhanced counting invariant.} Here the signature is
the cardinality of the image subbirack:
\[\Phi_{(X,B)}^{\mathrm{Im}}(L)=\sum_{\mathbf{w}\in (\mathbb{Z}_N)^c}
\left(\sum_{f\in \mathrm{Hom}(BR(L,\mathbf{w}),X)} u^{|\mathrm{Im}(f)|}\right);\]
\item
\textit{The writhe-enhanced counting invariant.} Here we keep track of which
writhe vectors contribute which labelings:
\[\Phi_{(X,B)}^{\mathrm{Im}}(L)=\sum_{\mathbf{w}\in (\mathbb{Z}_N)^c}
\left(\sum_{f\in \mathrm{Hom}(BR(L,\mathbf{w}),X)} 
|\mathrm{Hom}(BR(L,\mathbf{w}),X)|q^{\mathbf{w}}\right)\]
where $q^{(w_1,\dots,w_c)}=q_1^{w_1}\dots q_c^{w_c}$, and
\item \textit{The birack polynomial enhanced counting invariant.} Here the 
signature is the subbirack polynomial $\rho_{\mathrm{Im}(f)\subset X}$ of the 
image subbirack (see \cite{N} for more):
\[\Phi_{(X,B)}^{\mathrm{Im}}(L)=\sum_{\mathbf{w}\in (\mathbb{Z}_N)^c}
\left(\sum_{f\in \mathrm{Hom}(BR(L,\mathbf{w}),X)} 
u^{\rho_{\mathrm{Im}(f)\subset X}}\right),\]
and 
\item \textit{The column group enhanced counting invariant.} Here the 
signature is the subgroup $CG(\mathrm{Im}(f))$ of $S_{|X|}$ generated by 
the permuations $u_x$ and $l_x$ for $x\in\mathrm{Im}(f)$:
\[\Phi_{(X,B)}^{CG}(L)=\sum_{\mathbf{w}\in (\mathbb{Z}_N)^c}
\left(\sum_{f\in \mathrm{Hom}(BR(L,\mathbf{w}),X)} u^{|CG(\mathrm{Im}(f))|}\right).\]
See \cite{HN} for more.
\end{itemize}

As an application, we are able to answer a question posed to the second
author by Xiao-Sing Lin in 2005: can birack counting invariants be used to
distinguish non-invertible knots from their inverses? We are happy to say 
that the answer is yes, as we demonstrate in the next example.

\begin{example}
\textup{Consider the virtual knot numbered $3.3$ in the knot atlas \cite{KA};
it is the closure of the virtual braid diagram below. If we orient the
braid first downward and then upward, we have the listed fundamental birack 
presentations, $B(3.3_{\downarrow})$ and $B(3.3_{\uparrow})$ respectively.}
\[\includegraphics{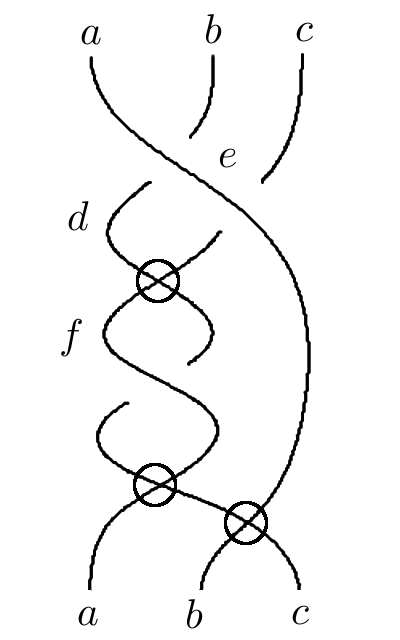} \quad \raisebox{1in}{$
\begin{array}{l}
B(3.3_{\downarrow})=\langle a,b,c,d,e,f \ |\ (b,a)=B(e,d),\ (c,e)=B(b,f),\ 
(d,f)=B(a,c) \rangle \\
 \ \\
 \ \\
B(3.3_{\uparrow})=\langle a,b,c,d,e,f \ |\ B(a,b)=(d,e),\ B(e,c)=(f,b),\ 
B(f,d)=(c,a) \rangle \\
\end{array}$}\]
\textup{Note that these are related by 
$B_{\downarrow}=\tau \circ B_{\uparrow}\circ \tau$, so both oriented versions of 
$3.3$ have the same fundamental involutory birack.
}

\textup{To see that the two oriented versions are non-isotopic, we will
need a non-involutory birack. Let $X=\mathbb{Z}_{11}$ and set $t=6$, $s=5$ 
and $r=3$. Then we have $s^2=5^2=(1-6(3))5=(1-tr)s$ and we have a 
$(t,s,r)$-birack; moreover, $t^2=6^2=3\ne 1$ and $X$ is non-involutory.
Since $\mathbb{Z}_{11}$ is a field, $s$ is invertible and $X$ is biquandle,
so $N=1$ and we do not need to compute labelings for multiple writhes.}

\textup{Then setting $B(x,y)=(5x+6y,3x)$, the crossing relations for 
$B_{\downarrow}$ give us a system of linear equations over $\mathbb{Z}_{11}$
with homogeneous matrix}
\[
\begin{array}{rcl}
b & = & 5e+6d \\
a & = & 3e \\
c & = & 5b+6f \\
e & = & 3b \\
d & = & 5a+6c \\
f & = & 3a \\
\end{array}
\Rightarrow
\left[\begin{array}{rrrrrr}
0 & 10 & 0 & 6 & 5 & 0 \\
10 & 0 & 0 & 0 & 3 & 0 \\
0 & 5 & 10 & 0 & 0 & 6 \\
0 & 3 & 0 & 0 & 10 & 0 \\
5 & 0 & 6 & 10 & 0 & 0 \\
3 & 0 & 0 & 0 & 0 & 10 \\
\end{array}\right]
\mathrm{\ is\ row\ equivalent \ to\ }
\left[\begin{array}{rrrrrr}
1 & 0 & 0 & 0 & 0 & 0 \\
0 & 1 & 0 & 0 & 0 & 0 \\
0 & 0 & 1 & 0 & 0 & 0 \\
0 & 0 & 0 & 1 & 0 & 0 \\
0 & 0 & 0 & 0 & 1 & 0 \\
0 & 0 & 0 & 0 & 0 & 1 \\
\end{array}\right]
\]
\textup{which has full rank, and hence the only solution is the trivial
labeling of all semiarcs by $0$, and the counting invariant is
$\Phi_{(X,B)}^{\mathbb{Z}}(3.3_{\downarrow})=u^1=u$. On the other hand, the 
crossing relations for $B(3.3_{\uparrow})$ give us the system of linear 
equations over $\mathbb{Z}_{11}$ with homogeneous matrix}
\[
\begin{array}{rcl}
d & = & 5a+6b \\
e & = & 3a \\
f & = & 5e+6c \\
b & = & 3e \\
c & = & 5f+6d \\
a & = & 3f \\
\end{array}
\Rightarrow
\left[\begin{array}{rrrrrr}
5 & 6 & 0 & 10 & 0 & 0 \\
3 & 0 & 0 & 0 & 10 & 0 \\
0 & 0 & 6 & 0 & 5 & 10 \\
0 & 10 & 0 & 0 & 3 & 0 \\
0 & 0 & 10 & 6 & 0 & 5 \\
10 & 0 & 0 & 0 & 0 & 3
\end{array}\right]
\mathrm{\ is\ row\ equivalent \ to\ }
\left[\begin{array}{rrrrrr}
1 & 0 & 0 & 0 & 0 & 8 \\
0 & 1 & 0 & 0 & 0 & 6 \\
0 & 0 & 1 & 0 & 0 & 0 \\
0 & 0 & 0 & 1 & 0 & 10 \\
0 & 0 & 0 & 0 & 1 & 2 \\
0 & 0 & 0 & 0 & 0 & 0
\end{array}\right]
\]
\textup{which has rank 1, and thus the space of solutions is $1$-dimensional,
giving us a counting invariant value of
$\Phi_{(X,B)}^{\mathbb{Z}}(3.3_{\uparrow})=u^{11}\ne u$, and the counting invariant
detects the non-invertibility of the virtual knot $3.3$.}
\end{example}

\section{\large\textbf{Questions}}\label{Q}

In this section we collect a few questions for future research.

What new enhancements of the counting invariant require $X$ to be 
involutory? Does the condition that $(X,B)$ is involutory determine
anything about the homology groups, column groups or birack polynomials 
of $(X,B)$?

In remark \ref{rmk1} we observed that the component maps of an
involutory birack must be involutions. Is the converse true? That is,
do the conditions $(\tau \circ B)^2=\mathrm{Id}$ and $S\circ B=\mathrm{Id}$
follow from the condition that $u_x$ and $l_x$ are involutions for 
all $x\in X$?

What conditions on two non-involutory finite biracks $B,B'$ imply that 
$I(B)\cong I(B')$? That is, when do two non-involutory birack involutize
to the same involutory birack?

\bigskip

\noindent
\textsc{Department of Mathematical Sciences \\
Claremont McKenna College \\
850 Columbia Ave. \\
Claremont, CA 91711} 

\bigskip

\noindent
\textsc{Dept. of Mathematics \\
5734 S. University Avenue \\
Chicago, Illinois 60637}

\end{document}